\documentclass[12pt]{article} 

\usepackage{amsmath}
\usepackage{amsthm}
\usepackage{amsfonts}
\usepackage{mathrsfs}
\usepackage{stmaryrd}
\usepackage{setspace}
\usepackage{fullpage}
\usepackage{amssymb}
\usepackage{breqn}
\usepackage{enumitem}
\usepackage{bbold} 
\usepackage{authblk}
\usepackage{comment}
\usepackage{hyperref}
\usepackage{pgf,tikz}
\usepackage{graphicx}
\usepackage{subcaption}

\newtheorem{thm}{Theorem}[section]

\newtheorem{lemma}[thm]{Lemma}
\newtheorem{conjecture}[thm]{Conjecture}

\newtheorem{prop}[thm]{Proposition}

\newtheorem{clm}[thm]{Claim}

\newtheorem*{theorem*}{Theorem}

\newcommand\ex{\ensuremath{\mathrm{ex}}}

\newcommand\cN{{\mathcal N}}

\usepackage{lineno}

\newcommand{\ignore}[1]{}

\title{On weakly Turán-good graphs}
\author{D\'aniel Gerbner\footnote{Alfr\'ed R\'enyi Institute of Mathematics, E-mail: \texttt{gerbner@renyi.hu.}}}
\date{}

\begin{document}

\maketitle

\begin{abstract} Given graphs $H$ and $F$ with $\chi(H)<\chi(F)$, we say that $H$ is weakly $F$-Tur\'an-good if among $n$-vertex $F$-free graphs, a $(\chi(F)-1)$-partite graph contains the most copies of $H$. Let $H$ be a bipartite graph that contains a complete bipartite subgraph $K$ such that each vertex of $H$ is adjacent to a vertex of $K$. We show that $H$ is weakly $K_3$-Tur\'an-good, improving a very recent asymptotic bound due to Grzesik, Gy\H ori, Salia and Tompkins. They also showed that for any $r$ there exist graphs that are not weakly $K_r$-Tur\'an-good. We show that for any non-bipartite $F$ there exists graphs that are not weakly $F$-Tur\'an-good. We also show
examples of graphs that are $C_{2k+1}$-Tur\'an-good but not $C_{2\ell+1}$-Tur\'an-good for every $k>\ell$.
\end{abstract}

\section{Introduction}

Given a graph $F$, $\ex(n,F)$ denotes the largest number of edges in $n$-vertex $F$-free graphs. Tur\'an \cite{T} proved that $\ex(n,K_{r+1})=|E(T(n,r))|$, where the \textit{Tur\'an graph} $T(n,r)$ is the complete $r$-partite graph with each part of order $\lfloor n/r\rfloor$ or $\lceil n/r\rceil$. Simonovits \cite{simi} proved that for an $(r+1)$-chromatic $F$ we have $\ex(n,F)=|E(T(n,r))|$ for sufficiently large $n$ if and only if $F$ has a color-critical edge, i.e., an edge whose removal decreases the chromatic number. The Erd\H os-Stone-Simonovits theorem \cite{ES1966,ES1946} states that for any $(r+1)$-chromatic graph $F$ we have $\ex(n,F)=|E(T(n,r))|+o(n^2)$.

Given graphs $H$ and $G$, let $\cN(H,G)$ denote the number of copies of $H$ in $G$. In \textit{generalized Tur\'an problems}, we deal with $\ex(n,H,F)$, which is the largest $\cN(H,G)$ where $G$ is an $n$-vertex $F$-free graph. The first result in this area is due to Zykov \cite{zykov}, who showed that $\ex(n,K_k,K_{r+1})=\cN(K_k,T(n,r))$.

Generalized Tur\'an problems have attracted several researchers since then. One of the main directions of research has been studying when the Tur\'an graph, or more generally a complete $(\chi(F)-1)$-partite graph is extremal. Given a graph $F$ with $\chi(F)=r+1$, we say that $H$ is \textit{$F$-Tur\'an-good} if $\ex(n,H,F)=\cN(H,T(n,r))$ for $n$ sufficiently large, and we say that $H$ is \textit{weakly $F$-Tur\'an-good} if $\ex(n,H,F)=\cN(H,T)$ for $n$ sufficiently large for some complete $r$-partite $n$-vertex graph $T$. Note that for a given $H$, it is straightforward but complicated to determine which $n$-vertex complete $r$-partite graph contains the most copies of $H$. 

Gy\H ori, Pach and Simonovits \cite{gypl} started the systematic study of $K_{r+1}$- Tur\'an-good graphs. They showed that if $H$ is a complete $k$-partite graph with $k\le r$, then $H$ is weakly $K_{r+1}$-Tur\'an-good. They also constructed several $K_{r+1}$-Tur\'an-good graphs. In particular, if $H$ is a bipartite graph with a matching containing all but at most one of its vertices, then $H$ is $K_3$-Tur\'an-good.

Gerbner and Palmer \cite{gerpal} initiated the study of $F$-Tur\'an-good graphs for non-complete graphs $F$. They also conjectured that paths are $K_{r+1}$-Tur\'an-good for any $r\ge 2$. This was proved in \cite{ger5}, after partial results in \cite{gerb,mn,qxg,hhl}.

We say that $H$ is \textit{asymptotically $F$-Tur\'an-good} if $\ex(n,H,F)=(1+o(1))\cN(H,T(n,r))$ and we say that $H$ is \textit{asymptotically weakly $F$-Tur\'an-good} if $\ex(n,H,F)=(1+o(1))\cN(H,T)$ for some complete $r$-partite graph $T$.
Let $H$ be a bipartite graph containing a subgraph $K$ isomorphic to $K_{s,t}$. Assume that each vertex $v\in V(H)$ is adjacent to a vertex of $V(K)$.
Grzesik, Gy\H ori, Salia and Tompkins \cite{ggyst} showed that $H$ is asymptotically weakly $K_3$-Tur\'an-good. We improve this to an exact result. Moreover, we extend the result to any 3-chromatic graph with a color-critical edge in place of $K_3$.

\begin{thm}\label{main}
Let $H$ be a bipartite graph containing a subgraph $K$ isomorphic to $K_{s,t}$ and assume that each vertex $v\in V(H)$ is adjacent to a vertex of $V(K)$. Let $F$ be a 3-chromatic graph with a color-critical edge. Then $H$ is weakly $F$-Tur\'an-good.
\end{thm}

Considering the large variety of weakly $K_{r+1}$-Tur\'an-good graphs, it is a natural idea that maybe all graphs of chromatic number at most $r$ have this property. However, Gy\H ori, Pach and Simonovits \cite{gypl} showed a bipartite graph that is not weakly $K_3$-Tur\'an-good, moreover, not even asymptotically weakly $K_3$-Tur\'an-good. Grzesik, Gy\H ori, Salia and Tompkins \cite{ggyst} showed for any $r\ge 2$ an $r$-chromatic graph that is not asymptotically weakly $K_{r+1}$-Tur\'an-good. Here we extend this.

\begin{prop}\label{main2}
For any non-bipartite $F$, there exists a graph of chromatic number $\chi(F)-1$ that is not asymptotically weakly $F$-Tur\'an-good.
\end{prop}


In extremal graph theory, graphs with a color-critical edge often behave similarly to cliques. We believe that this is the case in our setting as well.

\begin{conjecture} \label{coni}
If $F$ has a color-critical edge and chromatic number $r+1$, and $H$ is weakly $K_{r+1}$-Tur\'an-good, then $H$ is weakly $F$-Tur\'an-good.
\end{conjecture}

This conjecture is supported by the fact that the asymptotic version is true: if $H$ is asymptotically weakly $K_{r+1}$-Tur\'an-good, then $H$ is asymptotically weakly $F$-Tur\'an-good. In fact, $H$ is asymptotically weakly $F'$-Tur\'an-good for any $(r+1)$-chromatic graph $F'$. This follows from a theorem in \cite{GP}, stating that $\ex(n,H,F)\le \ex(n,H,K_{r+1})+o(n^{|V(H)|})$. However, the reverse is not true. In fact, for every $k>2$ we can construct a graph that is asymptotically $C_{2k+1}$-Tur\'an-good and not asymptotically weakly $K_3$-Tur\'an-good. We prove more. 

A \textit{blow-up} of a graph $G$ is obtained by replacing each vertex $v_i$ with a non-empty independent set $V_i$, and each edge $v_iv_j$ is replaced by all the possible edges between $V_i$ and $V_j$. The sets $V_i$ are called blown-up classes. We denote by $G(m)$ the blow-up where each $V_i$ has order $m$. A vertex $v$ of $G$ is \textit{color-critical} if the removal of $v$ decreases the chromatic number.


\begin{thm}\label{main3}
Let $F$ be a 3-chromatic graph and let $C_{2k+1}$ be the longest odd cycle such that a blow-up of it contains $F$. Then there is an asymptotically $F$-Tur\'an-good graph $H$ that is not asymptotically weakly $C_{2\ell+1}$-Tur\'an-good for any $\ell<k$. Furthermore, if $F$ has a color-critical vertex, then $H$ is $F$-Tur\'an-good.
\end{thm}

\section{Proofs}

We say that $H$ is \textit{$F$-Tur\'an-stable} if the following holds. If $G$ is an $n$-vertex $F$-free graph with $\cN(H,G)\ge \ex(n,H,F)-o(n^{|V(H)|})$, then $G$ can be obtained from $T(n,\chi(F)-1)$ by adding and removing $o(n^2)$ edges. We say that $H$ is \textit{weakly $F$-Tur\'an-stable} if the following holds. If $G$ is an $n$-vertex $F$-free graph with $\cN(H,G)\ge \ex(n,H,F)-o(n^{|V(H)|})$, then $G$ can be obtained from a complete $(\chi(F)-1)$-partite graph by adding and removing $o(n^2)$ edges. Note that it is equivalent to the property that $G$ can be turned into a $(\chi(F)-1)$-partite graph $G'$ by removing $o(n^2)$ edges. Indeed, 
if the second property holds but the first does not, then we need to add $\Omega(n^2)$ edges to turn $G'$ to a complete $(\chi(F)-1)$-partite graph $G''$. It is easy to see that we removed $o(n^{|V(H)|})$ copies of $H$ and then added $\Omega(n^{|V(H)|})$ copies of $H$. Indeed, each edge in $G''$ is clearly in $\Omega(n^{|V(H)|-2})$ copies of $H$. Therefore, $\cN(H,G)\le \cN(H,G'')-\Omega(n^{|V(H)|})$, a contradiction.

The well-known Erd\H os-Simonovits stability theorem \cite{erd1,erd2,sim} states that $K_2$ is  $F$-Tur\'an-stable for every $F$. Ma and Qiu \cite{mq} studied such stability in generalized Tur\'an problems first and showed that $K_k$ is $F$-Tur\'an-stable for every $F$ with chromatic number more than $k$. Hei, Hou and Liu \cite{hhl} and later Gerbner \cite{ger4} studied the connection of such stability and exact result more generally.
We will use two results from \cite{ger4}.

\begin{thm}[Gerbner \cite{ger4}]\label{stabi}

\textbf{(i)} If $H$ is weakly $K_r$-Tur\'an-stable, then $H$ is weakly $F$-Tur\'an-stable for any $r$-chromatic graph $F$. 

\textbf{(ii)} If
$F$ has a color-critical edge, then weakly $F$-Tur\'an-stable graphs are also weakly $F$-Tur\'an-good. 
\end{thm}

We will consider the \textit{double star} $S_{a,b}$. It consists of a central edge $uv$, $a$ leaves joined to $u$ and $b$ leaves joined to $v$. Gy\H ori, Wang and Woolfson \cite{gyww} proved that $S_{a,b}$ is weakly $K_3$-Tur\'an-good. Gerbner \cite{ger} showed that $S_{a,b}$ is weakly $F$-Tur\'an-good for any 3-chromatic graph $F$ with a color-critical edge. We show that $S_{a,b}$
is weakly $K_3$-Tur\'an-stable.

\begin{prop}\label{dswts}
$S_{a,b}$ is weakly $K_3$-Tur\'an-stable.
\end{prop}

\begin{proof} We let $f(x,y)=\binom{x-1}{a}\binom{y-1}{b}+\binom{y-1}{a}\binom{x-1}{b}$ if $a\neq b$ and $f(x,y)=\binom{x-1}{a}\binom{y-1}{b}$ if $a=b$.
Gy\H ori, Wang and Woolfson \cite{gyww} showed that in a $K_3$-free $n$-vertex graph $G$ with maximum degree $\Delta>n/2$, each edge is the central edge of at most $f(\Delta',n-\Delta')$ copies of $S_{a,b}$ for some $n/2\le \Delta'\le\Delta$. Moreover, $G$ has at most $\Delta(n-\Delta)$ edges. Let $F_2$ denote the graph consisting of two triangles sharing a vertex. Gerbner \cite{ger} showed that for any $\varepsilon'>0$ there exists $\delta'>0$ such that if an $F_2$-free $n$-vertex graph $G$ with maximum degree $\Delta\ge n/2$ has at least $\Delta(n-\Delta)-\delta' n^2$ edges, then $G$ can be turned into a bipartite graph by deleting at most $\varepsilon' n^2$ edges. We will pick a small $\delta'$.

Let $G$ be a $K_3$-free $n$-vertex graph with at least $\ex(n,S_{a,b},K_3)-\delta n^{a+b+2}$ copies of $S_{a,b}$ and we want to show that $G$ can be turned into a bipartite graph by deleting at most $\varepsilon n^2$ edges. We consider two cases based on the maximum degree of $G$. In each case, we argue that either
$G$ has sufficiently many edges to apply a previously established stability result or $G$ has too
few edges to be near extremal.

Assume first that $G$ has maximum degree $\Delta> n/2$. If the conclusion does not hold, then $G$ has at most $\Delta(n-\Delta)-\delta' n^2\le \Delta'(n-\Delta')-\delta' n^2$ edges by the previous paragraph. Each edge is the central edge  at most $q:=f(\Delta',n-\Delta')$ copies of $S_{a,b}$, thus there are at most $\Delta'(n-\Delta')q-\delta' n^2q=\cN(S_{a,b},K_{\Delta',n-\Delta'})-\delta'\alpha n^{a+b+2}$ copies of $S_{a,b}$ in $G$ for some $\alpha>0$ that depends on $a$ and $b$, but not on $\varepsilon$. Therefore, picking a sufficiently small $\delta'$, we obtain a contradiction with our assumption on $G$.

Assume now that $\Delta\le n/2$. In this case we will show that $G$ can be turned into $K_{\lfloor n/2\rfloor,\lceil n/2\rceil}$ by deleting at most $\varepsilon n^2$ edges. By the ordinary Erd\H os-Simonovits stability the conclusion holds unless $G$ has less than $n^2/4-\delta' n^2$ edges. Observe that each edge is the central edge of at most $f(\Delta,\Delta)\le f(\lfloor n/2\rfloor,\lfloor n/2\rfloor)\le f(\lfloor n/2\rfloor,\lceil n/2\rceil)=:q'$ copies of $S_{a,b}$. Therefore, there are at most $n^2q'/4-\delta' n^2q'\le \cN(S_{a,b},K_{\lfloor n/2\rfloor,\lceil n/2\rceil})-\delta'\alpha n^{a+b+2}$  copies of $S_{a,b}$ in $G$ for some $\alpha>0$ that does not depend on $\varepsilon$. We obtain a contradiction with our assumption on $G$ as in the previous paragraph.
\end{proof}

Note that this, combined with Theorem \ref{stabi}, gives a simpler proof of the theorem from \cite{ger} stating that $S_{a,b}$ is $F$-Tur\'an-good for every 3-chromatic graph $F$ with a color-critical edge.

\begin{prop}
Let $H$ be a bipartite graph containing a subgraph $K$ isomorphic to $K_{s,t}$ and assume that each vertex $v\in V(H)$ is adjacent to a vertex of $V(K)$. Then $H$ is weakly $K_3$-Tur\'an-stable.
\end{prop}

\begin{proof}
Recall that \cite{ggyst} showed that $H$ is asymptotically weakly $K_3$-Tur\'an-good. We follow their proof. First they showed that it is enough to consider $H$ that consists of $K$ and some pendant edges. Assume that $H$ has $a+1$ vertices in one of its parts and $b+1$ vertices in the other part. Let $G$ be a $K_3$-free $n$-vertex graph, let $p(G)$ denote the number of labeled copies of $H$ in $G$ and $q(G)$ denote the number of labeled copies of $S_{a,b}$ in $G$. It is shown in \cite{ggyst} that $p(G)\le q(G)+o(n^{a+b+2})$. Observe that in a complete bipartite graph the same ordered sets of $a+b+2$ vertices induce copies of $H$ and $S_{a,b}$, and obviously the ratio of the numbers of labeled and unlabeled copies of a graph depends only on the graph itself. These imply the asymptotic result.

Let us assume now that $G$ contains $\ex(n,H,K_3)-o(n^{a+b+2})=\cN(H,T)-o(n^{a+b+2})$ copies of $H$, for some $n$-vertex complete bipartite graph $T$. Then the number of labeled copies of $H$ also differs by $o(n^{a+b+2})$, i.e. $p(G)=p(T)-o(n^{a+b+2})$. Therefore, we have $q(G)=q(T)-o(n^{a+b+2})$, thus $\cN(S_{a,b},G)=\cN(S_{a,b},T)-o(n^{a+b+2})$. This, combined with Proposition \ref{dswts}, completes the proof.
\end{proof}

Combined with Theorem \ref{stabi}, the above proposition implies Theorem \ref{main}.

\smallskip

Let us continue with the proof of Proposition \ref{main2}.
Recall that it states that for any non-bipartite $F$, there exists a graph of chromatic number $\chi(F)-1$ that is not asymptotically weakly $F$-Tur\'an-good. Our construction is a slight generalization of the construction in \cite{ggyst}, which we describe next, after some necessary definition.

Given a graph $G$, $G^k$ denotes the graph we obtain by connecting two vertices of $G$ if and only if they are at distance at most $k$ in $G$.
The graph $H$ that is not weakly $K_{r+1}$-Tur\'an-good in \cite{ggyst} is obtained from $P_{2r+2}^{r-1}$ by replacing the end-vertices of the original path by sufficiently large independent sets. Then they show that a very unbalanced blow-up of $C_{2r+1}^{r-1}$ contains more copies of $H$ than any $r$-partite graph. On the other hand, any blow-up of $C_{2r+1}^{r-1}$ is $K_{r+1}$-free, since any set of $r+1$ vertices contains either 2 vertices from the same part of the blow-up, or 2 vertices that belong to parts that are at distance $r$ in the original $C_{2r+1}$.

For simplicity, in the following proof, we will count labeled copies of $H$ inside some host graphs. It is easy to see that it only gives a multiplicative factor (the number of automorphisms of $H$), independent of the host graph, thus does not affect our result.

\begin{proof}[Proof of Proposition \ref{main2}]

Let $\chi(F)=r+1$.
We are going to use the following graph $H$. We take $P_{k}^{r-1}$ and replace the end-vertices of the original path by independent sets of order $a$, where $a$ is  sufficiently large. We will show that a very unbalanced blow-up of $C_{k}^{r-1}$ contains more copies of $H$ than any $r$-partite graph. 

We show that if $k>|V(F)|$, then any blow-up of $C_{k}^{r-1}$ is $F$-free. Indeed, a copy of $F$ would avoid at least one of the $V_i$'s. But the remaining graph is $r$-colorable (hence $F$-free), as shown by the following coloring. Let us color the vertices inside each part by the same color, and color the parts the following way. We go through the original $C_k$ in a cyclic order, and assume $V_k$ is avoided by $F$. Then we color $V_j$ by color $j$ modulo $r$.

Observe that there is a unique $r$-coloring of any blow-up of $P_{k}^{r-1}$. In $H$, if $k$ is not congruent to 1 modulo $r$, then the two $a$-sets corresponding to the end vertices of the original path have different color. Therefore, inside an $r$-partite $n$-vertex graph $G$, those sets are in different parts. It implies that there at most $n^{k-2}\left(\frac{n}{2}\right)^{2a}$ labeled copies of $H$ in $G$.

Let $G'$ be the blow-up of $C_{k}^{r-1}$ where each vertex is replaced by $\lfloor \gamma n\rfloor$ vertices except one vertex is replaced by $n-(k-1)\lfloor \gamma n\rfloor$ vertices. Then the number of labeled copies of $H$ in $G'$ is at least $(\gamma n)^{k-2}(n-(k-1)\lfloor \gamma n\rfloor)^{2a}+o(n^{k-2+2a})$. For a given $\gamma<1/2k$, we can pick $a$ such that $\gamma^{k-2}(1-(k-1)\gamma)^{2a}>\frac{1}{2^{2a}}$, completing the proof.
\end{proof}

Let us turn to the proof of Theorem \ref{main3}. 
We start with a lemma.

\begin{lemma}\label{lemi} Let $F$ be a 3-chromatic graph with a color-critical vertex and let $C_{2k+1}$ be the longest odd cycle such that a blow-up of it contains $F$. Then $F$ is the subgraph of a blow-up of $C_{2k+1}$ where one of the blown-up classes has order 1.\end{lemma}

\begin{proof} Let $v$ be a color-critical vertex of $F$ and consider a blow-up $G$ of $C_{2k+1}$ with parts $V_1,\dots,V_{2k+1}$ in cyclic order, such that $G$ contains $F$. Assume that $v\in V_{2k+1}$. We pick $G$ such a way that $V_{2k+1}$ is as small as possible.

Let us assume that there is another vertex $v'\in V_{2k+1}$. If $v'$ is adjacent only to vertices in $V_1$, then we can move $v'$ to $V_2$, thus $|V_{2k+1}|$ decreases, a contradiction. If $v'$ is adjacent only to vertices in $V_{2k}$, then we can move $v'$ to $V_{2k-1}$, a contradiction. Assume that $v'$ has neighbors $v_1\in V_1$ and $v_{2k}\in V_{2k}$. 
Let $G'$ denote the bipartite graph we obtain from $G$ by deleting $V_{2k+1}$, and let $U$ denote the component containing $v_1$ in $G'$. Then we move every vertex of $U$ from $V_i$ to $V_{2k+1-i}$ for every $i$. We repeat this for every remaining neighbor of $v'$ in $V_1$. At the end, $v'$ has neighbors only in $V_{2k}$, thus we can move $v'$ to $V_{2k-1}$, a contradiction.
\end{proof}

Let $P_k$ denote the path on $k$ vertices.

\begin{prop}\label{turg1}
If $F$ is a 3-chromatic graph with a color-critical vertex, then for any $m\ge |V(F)|$ we have that $P_{2\ell}(m)$ is $F$-Tur\'an-good.
\end{prop}

We remark that the first result concerning $F$-Tur\'an-good graphs when $F$ does not have a color-critical edge is due to Gerbner and Palmer \cite{gerpal}, who showed that $C_4$ is $F_2$-Tur\'an-good, where $F_2$ consists of two triangles sharing a vertex. Gerbner \cite{ger2} constructed $F$-Tur\'an-good graphs for every $F$ with a color-critical vertex, but they were always complete $(\chi(F)-1)$-partite graphs. In particular, $K_{m,m}=P_2(m)$ is $F$-Tur\'an-good. The above proposition gives the first examples of another kind.

\begin{proof}
We apply induction on $\ell$, the base case $\ell=1$ was mentioned above. Assume the statement holds for $\ell$ and prove it for $\ell+1$. We count the copies of $P_{2\ell+2}(m)$ in an $n$-vertex $F$-free graph $G$ the following way. First we pick a copy of $P_{2\ell}(m)$, the number of ways to pick them is maximized when $G=T(n,2)$ by induction. Then, among the remaining $n-2\ell m$ vertices, we pick a copy of $P_2(m)=K_{m,m}$. The number of ways to pick it is maximized when there is a $T(n-2\ell m,2)$ on the remaining $n-2\ell m$ vertices, which is achieved when $G=T(n,2)$. 

We show that there are at most two ways to add the copy of $P_2(m)$ to the copy of $P_{2\ell}(m)$ in $G$ to create a copy of $P_{2\ell+2}(m)$. Indeed, we can do that if each vertex of a part of $K_{m,m}$ is adjacent to each vertex of one of the ends of $P_{2\ell}(m)$. It cannot happen with the same end of $P_{2\ell}(m)$ and both parts of $K_{m,m}$, as that would mean $G$ contains $K_{m,m,1}$, which contains $F$, a contradiction. In $T(n,2)$, there are always two ways to add a copy of $P_2(m)$ to a copy of $P_{2\ell}(m)$ in $T(n,2)$ to create a copy of $P_{2\ell+2}(m)$. Therefore, this third factor is also maximized by the Tur\'an graph, completing the proof.
\end{proof}

Let us denote by $P_{2k+2}(m,a,b)$ the following blow-up of $P_{2k+2}$. We replace the end vertices $v_1$ and $v_{2k+2}$ by independent sets $V_1$ of order $a$ and $V_{2k+2}$ of order $b$, and we replace the middle vertices $v_2,\dots,v_{2k+1}$ by independent sets $V_2,\dots,V_{2k+1}$ of order $m$.

\begin{prop}\label{turg2} Let $F$ be a 3-chromatic graph and assume that $F$ is contained in a blow-up of $C_{2k+1}$.
If $a\ge b\ge m\ge |V(F)|$ and $b\ge\binom{a-b}{2}$, then $P_{2k+2}(m,a,b)$ is asymptotically $F$-Tur\'an-good. Moreover, if $F$ has a color-critical vertex and $a<b+1/2+\sqrt{2b+1/4}$, then $P_{2k+2}(m,a,b)$ is $F$-Tur\'an-good.
\end{prop}

 As we have mentioned, $K_{a,b}$ is weakly $K_3$-Tur\'an-good by a result of Gy\H ori, Pach and Simonovits \cite{gypl}, thus only an optimization is needed here. Brown and Sidorenko \cite{brosid} did this optimization in a slightly different context, and obtained that $T(n,2)$ is asymptotically optimal if $b\ge\binom{a-b}{2}$.
Ma and Qiu \cite{mq} showed that $K_{a,b}$ is $K_3$-Tur\'an-good if and only if $a<b+1/2+\sqrt{2b+1/4}$.

\begin{proof}
Let $H_0$ denote the subgraph of $P_{2k+2}(m,a,b)$ obtained by deleting $V_1$ and $V_{2k+2}$, i.e. $H_0=P_{2k}(m)$. Then $H_0$ has a complete matching, thus is $K_3$-Tur\'an-good by a result of Gy\H ori, Pach and Simonovits \cite{gypl}. This implies that $H_0$ is asymptotically $F$-Tur\'an-good by a result of Gerbner and Palmer \cite{GP} mentioned in the introduction. 
We show that the number of ways to extend $H_0$ to $P_{2k+2}(m,a,b)$ in $G$ is also asymptotically at most the number of ways to extend $H_0$ to $P_{2k+2}(m,a,b)$ in the Tur\'an graph. 

Let us pick a copy of $H_0$ in $G$ and let $U$ denote the set of its vertices. Observe that there are at most $m-1$ vertices in $G$ that are not in $H_0$ and are connected to all the vertices of $U\cap V_2$ and $U\cap V_{2k+1}$, as otherwise $G$ contains $C_{2k+1}(m)$, which contains $F$. Let $x$ be the number of common neighbors of the vertices of $V_2$ that are not in $U$, then there are at most $n-2km+m-1-x$ common neighbors of the vertices of $V_{2k+1}$ that are not in $U$. We need to pick $a$ common neighbors of the $m$ vertices in $V_2$, and $b$ common neighbors of the $m$ vertices in $V_{2k+1}$, or the other way around, thus there are at most $\binom{x}{a}\binom{n-2km+m-1-x}{b}+\binom{x}{b}\binom{n-2km+m-1-x}{a}$ ways to extend $H_0$ to $P_{2k+2}(m,a,b)$ if $a\neq b$, and half of this if $a=b$. Observe that this is equal to the number of copies of $K_{a,b}$ in $K_{x,n-2km+m-1-x}$, which is at most the number of copies of $K_{a,b}$ in $T(n-2km+m-1,2)$, using the theorem of Brown and Sidorenko we mentioned. It is easy to see that $\cN(K_{a,b},T(n-2km+m-1,2))=(1+o(1))\cN(K_{a,b},T(n-2km,2))$. Therefore, the Tur\'an graph asymptotically maximizes the number of ways to extend $H_0$ to $H$, completing the proof for general $F$.
Let us assume now that $F$ has a color-critical vertex $v$. Recall that by Lemma \ref{lemi}, $F$ is contained in a blow-up of $C_{2k+1}$ where one part has order 1. 
We use this analogously to the argument before. We pick a copy of $H_0$ in $G$, then there are no vertices that are not in $H_0$ and are connected to all the vertices in $V_2$ and $V_{2k+1}$.
Let $x$ be the number of common neighbors of $V_2$ that are not in the copy of $H$.
We need to pick $a$ common neighbors of the $m$ vertices in $V_2$, and $b$ common neighbors of the $m$ vertices in $V_{2k+1}$, or the other way around. There are at most $\binom{x}{a}\binom{n-2km}{b}+\binom{x}{b}\binom{n-2km}{a}$ ways to do this if $a\neq b$ and half of this if $a=b$. Observe that this is equal to the number of copies of $K_{a,b}$ in $K_{x,n-2km-x}$, which is at most the number of copies of $K_{a,b}$ in $T(n-2km,2)$ by the theorem of Ma and Qiu we mentioned. This shows that the Tur\'an graph maximizes the number of ways to extend $H_0$ to $H$, completing the proof,
\end{proof}

Now we are ready to prove Theorem \ref{main3} which we restate here for convenience.

\begin{theorem*}
Let $F$ be a 3-chromatic graph and let $C_{2k+1}$ be the longest odd cycle such that a blow-up of it contains $F$. Then there is an asymptotically $F$-Tur\'an-good graph $H$ that is not asymptotically weakly $C_{2\ell+1}$-Tur\'an-good for any $\ell<k$. Furthermore, if $F$ has a color-critical vertex, then $H$ is $F$-Tur\'an-good.
\end{theorem*}

\begin{proof} 
Proposition \ref{turg2} created an asymptotically $F$-Tur\'an-good graph $H=P_{2k+2}(m,a,b)$ (which is an $F$-Tur\'an-good graph if $F$ has a color-critical vertex). We will show that $H$ is not asymptotically weakly $C_{2\ell+1}$-Tur\'an-good. Observe that the $a$ vertices and the $b$ vertices corresponding to the end vertices of the original $P_{2k+2}$ must be in different parts of any bipartite graph. On the other hand, they can be in the same part of blow-ups of $C_{2k+1}$, which are $C_{2\ell+1}$-free graphs. 
From here, the calculation to show that $H$ is not  asymptotically weakly $F'$-Tur\'an-good is the same as the calculation in the proof of Proposition \ref{main2}, thus we omit the details. 
\end{proof}

\section{Concluding remarks}

We have shown examples of graphs that are not asymptotically weakly $F$-Tur\'an-good, for any non-bipartite graph $F$. Our examples, just like all the earlier examples for the case $F$ is a clique, are built on the idea of forcing two large sets of vertices into different part in every $(\chi(F)-1)$-partite graph. In particular, the examples themselves are $(\chi(F)-1)$-chromatic. We do not know any examples of graphs that are not $F$-Tur\'an-good and have chromatic number less than $\chi(F)-1$ in the case $F$ has a color-critical edge.


Some of our results deal with asymptotically weakly $F$-Tur\'an-good graphs, while other results deal with weakly $F$-Tur\'an-good graphs. In the case $F$ has a color-critical edge, we are not aware of any graph that is asymptotically weakly $F$-Tur\'an-good but not weakly $F$-Tur\'an-good. In fact, the situation is much worse: in each of the cases when we can show that $H$ is not weakly $F$-Tur\'an-good, i.e., we can find an $F$-free $n$-vertex graph that contains more copies of $H$ than any $(\chi(F)-1)$-partite $n$-vertex graph, then we do not actually know $\ex(n,H,F)$. There are constructions with more copies of $H$ then any $(\chi(F)-1)$-partite $n$-vertex graph, but we do not know whether they are extremal.


Let us recall that Theorem \ref{main3} contains the assumption that $C_{2k+1}$ is the longest odd cycle such that the blow-up of it contains $F$. Then $F$ cannot contain any odd cycle of length less than $2k+1$. Then blow-ups of $F$ do not contain such cycles either. 
Thus Theorem \ref{main3} is a special case of the following conjecture.

\begin{conjecture}
For given $r+1$-chromatic graphs $F,F'$, there exists a graph $H$ that is asymptotically weakly $F$-Tur\'an-good but not asymptotically weakly $F'$-Tur\'an-good if and only $F'$ is not a subgraph of any blow-up of $F$. 
Furthermore, if $F$ has a color-critical vertex, then we can find $H$ that is $F$-Tur\'an-good.
\end{conjecture}

Note that one of the directions easily follows from known results. Let $H$ be asymptotically weakly $F$-Tur\'an-good. A result of Alon and Shikhelman \cite{ALS2016} states that if $F'$ is a subgraph of a blow-up of $F$, then $\ex(n,F,F')=o(n^{|V(F)|})$. Combined with the removal lemma \cite{efr}, this shows that we can delete the copies of $F$ from any $F'$-free $n$-vertex graph $G$ by deleting $o(n^2)$ edges, thus $o(n^{|V(H)|})$ copies of $H$. The resulting graph is $F$-free, thus contains at most $(1+o(1))\cN(H,T)$ copies of $H$ for some complete $r$-partite graph $T$, thus $G$ also contains at most $(1+o(1))\cN(H,T)$ copies of $H$, showing that $H$ is asymptotically weakly $F'$-Tur\'an-good.

\bigskip

\textbf{Funding}: Research supported by the National Research, Development and Innovation Office - NKFIH under the grants KH 130371, SNN 129364, FK 132060, and KKP-133819.

\end{document}